\documentclass[12pt]{amsart}
\usepackage{hyperref}
\usepackage{bbm}
\usepackage{enumitem}
\usepackage{nicefrac}
\usepackage{tikz}
\usepackage{amsmath}
\usepackage{mathtools}
\usepackage{tikz-cd}
\usepackage{amssymb}
\usepackage{dynkin-diagrams}
\usepackage[capitalise]{cleveref}
\usepackage[T1]{fontenc}%

\usepackage{lmodern}
\usetikzlibrary{calc,shapes, backgrounds,arrows,positioning,decorations.pathmorphing}
\newtheorem{theorem}{Theorem}[section]
\newtheorem{proposition}[theorem]{Proposition}
\newtheorem{lemma}[theorem]{Lemma}

\newtheorem*{claim*}{Claim}
\newtheorem{corollary}[theorem]{Corollary}
\newtheorem{Main Conjecture}[theorem]{Main Conjecture}
\newtheorem{conjecture}[theorem]{Conjecture}

\theoremstyle{definition}
\newtheorem{definition}[theorem]{Definition}
\newtheorem{construction}[theorem]{Construction}
\theoremstyle{remark}

\newtheorem{example}[theorem]{Example}

\theoremstyle{plain}

\hypersetup{%
   breaklinks,%
   colorlinks=true,%
   urlcolor=[rgb]{0.2,0.0,0.0},%
   linkcolor=[rgb]{0.5,0.0,0.0},%
   citecolor=[rgb]{0,0,0.445},%
   filecolor=[rgb]{0,0,0.4},%
   anchorcolor=[rgb]={0.0,0.1,0.2}%
}%

\newcommand{\Mod}[1]{\ (\text{mod}\ #1)}

\renewcommand{\Pr}[1]{\mathrm{Pr}\left[#1\right]}

\newcommand{\sPrg}[2]{\mathrm{Pr}[#1\,\,|\,\,#2]}
\newcommand{\E}[1]{\mathbb{E}\left[#1\right]}
\newcommand{\sEg}[2]{\mathbb{E}[#1\,\,|\,\,#2]}

\newcommand{\fun}[3]{#1 \colon #2 \to #3}

\renewcommand{\dim}[1]{\ensuremath{\mathsf{dim}\!\left(#1\right)}}

\newcommand{\Exp}[1]{\exp\left(#1\right)}

\newcommand{\inv}[1]{\left(#1\right)^{-1}}
\newcommand{\iv}[1]{#1^{-1}}

\let\ga=\alpha
\let\gb=\beta
\let\gc=\gamma
\let\gd=\delta

\let\gf=\varphi

\let\gi=\iota

\let\gn=\nu

\let\gp=\pi

\let\gC=\Gamma
\let\gD=\Delta

\let\gO=\Omega

\newcommand\numberthis{\addtocounter{equation}{1}\tag{\theequation}}

\let\e\varepsilon               % a ``real'' epsilon

\let\Z\Integer

\def\Floor#1{\left\lfloor #1 \right\rfloor}
\def\Ceil#1{\left\lceil #1 \right\rceil}

\def\Set#1{\left\{ #1 \right\}}

\title{Proper elements of Coxeter groups}
\author{J\'{o}zsef Balogh}
\address{Department of Mathematics, U. of Illinois at Urbana-Champaign, Urbana, IL 61801, USA}
\email{jobal@illinois.edu}
\author{David Brewster}
\address{John A. Paulson School of Engineering and Applied Sciences, Harvard University, Cambridge, MA 02138}
\email{dbrewster@g.harvard.edu}
\author{Reuven Hodges}
\address{Department of Mathematics, U. of Kansas, Lawrence, KS 66045, USA}
\email{rmhodges@ku.edu}

\begin{document}

\begin{abstract}
  We extend the notion of \emph{proper elements} to all finite Coxeter groups.
  For all infinite families of finite Coxeter groups we prove that the probability a
  random element is \emph{proper} goes to zero in the limit. This proves a conjecture of the third author and A.~Yong regarding the proportion of Schubert varieties that are Levi spherical for all infinite families of Weyl groups. We also enumerate the proper elements in the exceptional Coxeter groups.
\end{abstract}

\maketitle
\section{Introduction}

The study of proper elements is motivated by the study of reductive group
actions on Schubert varieties.
Let $G$ be a complex, connected reductive group, with a fixed maximal torus $T$
contained in a fixed Borel subgroup $B$.
The \emph{Weyl group} of $G$ is defined to be $W = N(T) / T$;
it is a finite Coxeter group of rank $n$, where $n$ is the semisimple rank of $G$.
The \emph{flag variety} $G/B$ is an object of central importance in algebraic
geometry and representation theory.
The $B$-orbits in $G/B$ are indexed by $w \in W$, and the Zariski-closure of the $B$-orbit indexed by $w$ is the
\emph{Schubert variety} $X_w$.

The standard Levi subgroups of $G$ are families of reductive subgroups that act on
Schubert varieties in $G/B$.
For each $I \subseteq [n]:=\{1,\ldots,n\}$, there is an associated
standard parabolic subgroup $P_I \supseteq B$.
Each $P_I$ decomposes as a semidirect product \[P_I = L_I \ltimes U_I,\] where
$L_I$ is a reductive group called a standard Levi subgroup and $U_I$ is the unipotent
radical of $P_I$. 

The group $G$ acts on $G/B$ by left multiplication. If $J(w) \subseteq [n]$ is the left descent set
of $w$ (see Definition \ref{def:left-descents}), then ${\sf stab}_G(X_w) = P_{J(w)}$~\cite[Lemma~8.2.3]{BL00}. For any $I \subseteq J(w)$,
$L_I < P_I \leq P_{J(w)}$ and hence $L_I$ is a reductive group that acts on $X_w$ by left multiplication.

A normal variety $X$ is a \emph{spherical variety} for the action of a reductive group $R$ if a Borel
subgroup of $R$ has an open dense orbit in $X$. In \cite{HY20}, the third author and Alexander Yong initiated a study of when a
Schubert variety in $G/B$ is \emph{Levi spherical}; that is, when it is a spherical variety under the left multiplication action of a
standard Levi subgroup of $G$.

In this work we define proper elements of a Coxeter group and show that if $X_w$ is $L_I$-spherical, then $w$ is proper. We then analyze the limiting behavior of properness which yields a proof of ~\cite[Conjecture~3.7]{HY20} for all infinite families of Weyl groups.

\begin{theorem}\label{thm:main2}
  Let $G$ be a simple group with Weyl group $W$ of type $A_n, B_n, C_n,$ or
  $D_n$.
  Let $w$ be sampled uniformly at random from $W$.
  Then as $n\longrightarrow\infty$, \[
    {\rm Pr}[X_w \subseteq G/B \text{\ is $L_{J(w)}$-spherical}]
    \longrightarrow 0.
  \]
\end{theorem}

\subsection{Proper elements in finite Coxeter groups}

An $n\times n$ matrix $M$ is a \emph{Coxeter matrix} if it is a symmetric matrix with entries in $\{1,2,\ldots,\infty\}$ such that $M_{ij}=1$ if and only if $i=j$.
The $\emph{Coxeter group}$ associated to $M$ is the group
\[
W = \langle s_1,\ldots,s_n : (s_i s_j)^{M_{ij}} = e\text{ for all }M_{ij}\neq \infty \rangle.
\]
The matrix $M$ is visualized by the \emph{Coxeter diagram} ${\mathcal G}$, a graph whose nodes are labeled by
$[n]$, and nodes $i$ and $j$ are connected by an edge labeled by $M_{ij}$ if $M_{ij} \geq 3$. By convention, the label on edges with $M_{ij} = 3$ is omitted. See Figure~\ref{fig:coxG} below for examples.

Let $S = \{ s_1, \ldots, s_n \}$. The pair $(W,S)$ is called a \emph{Coxeter system}. If $W$ is a finite group, then $(W,S)$ is \emph{finite}; $(W,S)$ is \emph{irreducible} if ${\mathcal G}$ is connected. The finite Coxeter groups were classified in ~\cite{C35}. The irreducible, finite Coxeter groups consist of four infinite families $A_n, B_n, D_n,$ and $I_2(n)$ as well as 6 exceptional groups: $E_6, E_7, E_8, F_4, H_3,$ and $H_4$. The \emph{Coxeter length}, $\ell(w)$, of $w \in W$ is equal to the minimal number of elements of $S$ required to express $w$.
There is a unique element of maximal length in $W$ denoted by $w_0(W)$.

We index the nodes in the Coxeter diagram ${\mathcal G}$ by $[n]$.
For $I \subseteq [n]$, let ${\mathcal G}_I$ be the induced subdiagram of
${\mathcal G}$.
There is a decomposition of ${\mathcal G}_I$ into $m$ connected components
\begin{equation}\label{eqn:thedecomp}
  {\mathcal G}_I=\bigcup_{z=1}^m {\mathcal C}^{(z)}.
\end{equation}
where each ${\mathcal C}^{(z)}$ is a Coxeter diagram with associated Coxeter group $W^{(z)}$.
Let $W_I$ be the parabolic subgroup of $W$ generated by
$S_I := \{s_i \mid i \in I\}$.
Then ${\mathcal G}_I$ is the Coxeter diagram of $W_I$, and
\begin{equation}\label{eqn:thelength}
  \ell(w_0(W_I)) = \sum_{z=1}^m \ell(w_0(W^{(z)})).
\end{equation}

\begin{definition}\label{def:left-descents}
  For an element $w\in W$, the set of \emph{left descents} is\[
    J(w) := \{j \in [n]  \mid \ell(s_j w) < \ell(w) \}.
  \]
\end{definition}
The number of left descents will be denoted by $d(w) := |J(w)|$. For a nonnegative integer $x\!\leq\!n$ define ${\sf maxw}_0(W, x)\!:=\!\max\{\ell(w_0(W_I)) \mid I\!\subseteq\![n]\text{ and }|I|\!=\!x\}$.

\begin{definition}
\label{def:main}
  An element $w \in W$ is \emph{proper} if $\ell(w) \leq n + {\sf maxw}_0(W, d(w))$.
\end{definition}
\begin{example}
Let $W$ be the $B_3$ Coxeter group with $w = s_3 s_2 s_3 s_1 s_2 s_3 s_1 \in W$. Then $J(w) = \{ 2,3\}$ and $d(w)=2$. For $I = \{2,3 \} \subseteq [3]$, $W_I$ is the $B_2$ Coxeter group with $\ell(w_0(W_I))=4$. This $I$ achieves the maximum possible value for $\ell(w_0(W_I))$ over all $I\subseteq [3]$ with$|I|=d(w)=2$. Hence, ${\sf maxw}_0(W, d(w))=4$. We conclude \[\ell(w)=7 \leq 3 + 4 = n + {\sf maxw}_0(W, d(w)),\] and so $w$ is proper.
\end{example}
\begin{example}
Let $W$ be the $A_4$ Coxeter group with $w = s_2 s_3 s_4 s_1 s_2 s_3 s_1 s_2 \in W$. Then $J(w) = \{ 2,3\}$, $d(w)=2$, and ${\sf maxw}_0(W, d(w))=3$ by Proposition~\ref{prop:maxwAnalysis} below. Hence, \[ \ell(w)=8 \nleq 4 + 3 = n + {\sf maxw}_0(W, d(w)), \] and so $w$ is not proper.
\end{example}

We analyze the limiting behavior of this property.

\begin{theorem}\label{thm:main1}
  Let $W$ be a Coxeter group of type $A_n$, $B_n$, $D_n$, or $I_2(n)$.
  Let $w$ be sampled uniformly at random from $W$.
  Then as $n\longrightarrow\infty$, \[
    {\rm Pr}[w \text{\ is proper}]\longrightarrow 0.
  \]
\end{theorem}

In Proposition~\ref{prop:geosphericalimpliesproper}, it is shown that if $X_w \subseteq G/B$ is $L_I$-spherical for $I \subseteq J(w)$, then $w$ is proper. Hence, Theorem~\ref{thm:main1} implies Theorem~\ref{thm:main2}.

Section \ref{sec:exceptional} enumerates the proper elements for $W$ an
exceptional finite Coxeter group. In Theorem~\ref{thm:lowerBound}, we give a non-trivial lower bound on the number of proper elements in Coxeter groups of type $A_n$, $B_n$, and $D_n$.

\subsection{Classifying Levi-spherical Schubert varieties} A type-independent classification of Levi-spherical Schubert varieties was
conjectured in~\cite[Conjecture 1.9]{HY20} by the third author and A.~Yong. This proposed classification is further motivated by its connection to the theory of \emph{Demazure characters}, or \emph{key polynomials}, and the study of their ``split-symmetry''~\cite[\S 4.1]{HY20}. 

The proposed classification of Levi-spherical Schubert varieties in \cite{HY20}
is in terms of \emph{spherical elements}.
A \emph{reduced expression} of $w\in W$ is a word
$s_{i_1} \ldots s_{i_{\ell(w)}} = w$.
Denote the set of reduced words of $w$ by
${\text{Red}}(w):={\text{Red}}_{(W,S)}(w)$.

\begin{definition}[$I$-spherical elements]\label{def:sphricalCoxeter}
  Let $w\in W$ and fix $I\subseteq J(w)$. Then $w$ is \emph{$I$-spherical} if
  there exists $s_{i_1}\cdots s_{i_{\ell(w)}}\in {\text{Red}}(w)$ such that:
  \begin{enumerate}[label=(\textbf{S.\arabic*}),ref=S.\arabic*]
    \item \label{S.1} $\#\{t: i_t= j\}\leq 1$ for all $j \in [n]- I$, and
    \item \label{S.2} $\#\{t: i_t \in {\mathcal C}^{(z)}\} \leq \ell(w_0({W^{(z)}})) +\#\text{vertices}({\mathcal C}^{(z)})$
      for $1\leq z\leq m$.
  \end{enumerate}
\end{definition}

\begin{example}
Let $W$ be the $E_7$ Coxeter group. The $E_7$ Coxeter diagram is
\[ \dynkin[label, edge length=0.5cm]E7. \] Let $w = s_4 s_2 s_3 s_4 s_3 s_2 s_5 s_7 s_4 s_3 s_2 s_4 s_1 s_5 s_3 \in W$. Then $J(w)=\{2,3,4,5, 7\}$. If $I=J(w)$ then 
\[{\mathcal C}^{(1)}=\dynkin[labels={3,4,2,5}, edge length=0.5cm]D4\text{ and  }{\mathcal C}^{(2)}=\dynkin[labels={7}, edge length=0.5cm]A1.\]
Here $W^{(1)}$ is the $D_4$ Coxeter group with longest element $w_0(W^{(1)})=s_3 s_2 s_4 s_3 s_2 s_4 s_5 s_4 s_3 s_2 s_4 s_5$ and $\ell(w_0(W^{(1)}))=12$. On the other hand, $W_{I^{(2)}}$ is the $A_1$ Coxeter group with longest element
$w_0(W^{(2)})=s_7$ and $\ell(w_0(W^{(2)}))=1$. The reduced word given to define $w$ satisfies (S.1) since $s_1$ appears once and $s_6$ appears zero times. This reduced word also satisfies (S.2) since $s_2, s_3, s_4,$ and $s_5$ appear 13 times, which is less than $12 + 4 = 16$ and $s_7$ appears 1 time, which is less than $1 + 1 = 2$. Thus $w$ is $I$-spherical.

Proposition~\ref{prop:coxsphericalimpliesproper} below implies that since $w$ is $I$-spherical it is also proper. This is easily verified to be the case. We have $\ell(w)=15$. And this is less than $n + {\sf maxw}_0(W, d(w)) = 7 + {\sf maxw}_0(W, 5) = 27$, where the final equality follows from Proposition~\ref{prop:maxwAnalysis} below.
\end{example}
\begin{example}
We also give an example of an element that is not $I$-spherical. Let $W$ be the $B_3$ Coxeter group. The Coxeter diagram is \dynkin[Coxeter,label, edge length=0.5cm]B3. The element $w = s_3 s_2 s_3 s_1 s_2 s_3 \in W$ has $J(w)=\{3 \}$. Let $I=J(w)$. This $w$ has only two reduced words, $s_3 s_2 s_3 s_1 s_2 s_3$ and $s_3 s_2 s_1 s_3 s_2 s_3$. Each of these reduced words fail (S.1); in both, $s_2$ appears $2$ times which is greater than 1. Thus $w$ is not $I$-spherical. 
\end{example}

\begin{conjecture}[{\cite{HY20}}]\label{conj:main}
  Let $I\subseteq J(w)$.
  $X_w$ is ${L}_I$-spherical if any only if $w$ is $I$-spherical.
\end{conjecture}

The third author, joint with Y.~Gao and A.~Yong, proved Conjecture~\ref{conj:main} in type $A_n$~\cite{GHY21}. This was then used by C.~Gaetz in \cite{G21} to prove \cite[Conjecture 3.8]{HY20}, giving a pattern avoidance criterion for a Schubert variety to be Levi-spherical in type $A_n$. The pattern avoidance criterion, in combination with the Marcus-Tardos theorem, implies Theorem~\ref{thm:main2} in type $A_n$. It is an open question if a pattern avoidance criterion exists for a Schubert variety to be Levi-Spherical in types $B_n$, $C_n$, and $D_n$. 
\begin{theorem}\label{thm:main3}
  Let $W$ be a Coxeter group of type $A_n$, $B_n$, $D_n$, or $I_2(n)$.
  Let $w$ be sampled uniformly at random from $W$.
  Then as $n\longrightarrow\infty$, \[
    {\rm Pr}[w\text{\ is $J(w)$-spherical}]
    \longrightarrow 0.
  \]
\end{theorem}
Proposition~\ref{prop:coxsphericalimpliesproper} below shows that $w \in W$ is $I$-spherical implies $w$ is proper. Hence Theorem \ref{thm:main3} follows from Theorem \ref{thm:main1}.

\subsection{Previous results on proper permutations}
\label{subsec:prev}
Proper permutations were first introduced in \cite{BHY20}. We highlight that our definition differs slightly from the original definition. The original definition of properness was motivated by the study of Levi-spherical Schubert varieties in $GL_n/B$. To study Levi-spherical Schubert varieties in $G/B$, for $G$ a simple Lie group, requires a definition of properness that corresponds to Levi-spherical Schubert varieties in $SL_n/B$ ($SL_n$ being a simple Lie group). This introduces a difference of $1$ on the right hand side of Definition~\ref{def:main} as compared to \cite[Definition 1]{BHY20}; this is due to the fact that the Levi subgroups in $SL_n$ have dimension one less than the corresponding Levi subgroups in $GL_n$. This updated definition is considerably more natural in the general type setting. 

The upper bound achieved for the number of proper permutations in \cite{BHY20} uses Chebyshev's inequality.
In this paper, we apply Chernoff bounds to achieve much tighter bounds, exponentially better than those in \cite{BHY20}.
We apply these techniques not only in type $A_n$, but also in types $B_n$ and $D_n$.

We now describe the layout of this paper. In Section \ref{sec:prep} explicit formulas for ${\sf maxw}_0(W, x)$ are given for each type. Next, it is shown that for an element $w$ in a Weyl group $W$, $X_w$ being $L_I$-spherical implies that $w$ is proper. And for $W$ a finite Coxeter group with $w \in W$, $w$ is $I$-spherical implies that $w$ is proper. Hence Theorem~\ref{thm:main1} implies Theorem~\ref{thm:main2} and Theorem~\ref{thm:main3}. In Section \ref{sec:prob} we derive concentration bounds that will be used to bound the number of proper elements in Coxeter groups of type $A_n$, $B_n$, and $D_n$. Section \ref{sec:combmodels} recalls several well-known combinatorial models for the Coxeter groups of type $A_n$, $B_n$, and $D_n$. Theorem~\ref{thm:main1} is proved in Section \ref{section:5} by giving asymptotic bounds on the number of proper elements. In Section~\ref{sec:lowerbounds} we give nontrivial lower bounds for the number of proper elements in Coxeter groups of type $A_n$, $B_n$, and $D_n$. We conclude with a table, presented in Section~\ref{sec:exceptional}, enumerating the number of proper elements in each of the exceptional finite Coxeter groups.

\section{Preparation}
\label{sec:prep}

\subsection{An analysis of ${\sf maxw}_0(W, x)$}
\begin{figure}\renewcommand{\arraystretch}{1.55}\begin{tabular}{rrlll}
Type && Diagram & Length of longest element & Number of elements \\
\hline
$A_n$ & \qquad & \dynkin[Coxeter]{A}{} & $\binom{n+1}{2}$ & $(n+1)!$ \\
$B_n, C_n$ && \dynkin[Coxeter]{B}{} &$n^2$&$2^nn!$\\
$D_n$ && \dynkin[Coxeter]{D}{} &$n^2-n$&$2^{n-1}n!$\\
$E_6$ && \dynkin[Coxeter]{E}{6}&36& 51,840 \\
$E_7$ && \dynkin[Coxeter]{E}{7}&63& 2,903,040 \\
$E_8$ && \dynkin[Coxeter]{E}{8}\qquad \,&120 & 696,729,600\\
$F_4$ && \dynkin[Coxeter]{F}{4}&24& 1,152 \\
$H_3$ && \dynkin[Coxeter]{H}{3}&15& 120\\
$H_4$ && \dynkin[Coxeter]{H}{4}&60& 14,400\\
$I_2(n)$ && \dynkin [Coxeter,gonality=n]I{}&$n$&$2n$ \\
\end{tabular}\caption{The finite Coxeter groups~\cite{H90}.}\label{fig:coxG}\end{figure}

We begin with a study of ${\sf maxw}_0(W, x)$ in Coxeter groups of each type,
in preparation for probabilistic analysis in later sections.

\begin{proposition}
\label{prop:maxwAnalysis}
  Let (W,S) be a finite Coxeter system.
  \begin{itemize}
    \item[(i)]
      If W is of type $A_n$, then ${\sf maxw}_0(W, x) = {x + 1 \choose 2}$.
    \item[(ii)]
      If $W$ is of type $B_n$ or $C_n$, then  ${\sf maxw}_0(W, x) = x^2$.
    \item[(iii)]
      If $W$ is of type $D_n$, then \[
        {\sf maxw}_0(W, x) :=
        \begin{cases}
          x^2 - x & x > 3 \\
          {x + 1 \choose 2} & x \leq 3
        \end{cases}.
      \]
    \item[(iv)]
      If $W$ is of type $E_6, E_7,$ or $E_8$, then \[
        {\sf maxw}_0(W, x) :=
        \begin{cases}
          120 & x=8 \\
          63 & x=7 \\
          36 & x=6 \\
          20 & x=5 \\
          12 & x=4 \\
          6 & x=3 \\
          3 & x=2 \\
          1 & x=1 \\
          0 & x=0
        \end{cases}.
      \]
    \item[(v)]
      If $W$ is of type $F_4$, then  \[
        {\sf maxw}_0(W, x) :=
        \begin{cases}
          24 & x=4 \\
          9 & x=3 \\
          4 & x=2 \\
          1 & x=1 \\
          0 & x=0
        \end{cases}.
      \]
%    \item[(vi)]
%      If $W$ is of type $G_2$, then $w \in W$ is proper if and only if
%      $\ell(w) \leq n + m_G(d(w))$ where \[
%        m_G(x) :=
%        \begin{cases}
%          6 & x=2 \\
%          1 & x=1 \\
%          0 & x=0
%        \end{cases}
%      \]
    \item[(vi)]
      If $W$ is of type $H_3$ or $H_4$, then  \[
        {\sf maxw}_0(W, x) :=
        \begin{cases}
          60 & x=4 \\
          15 & x=3 \\
          5 & x=2 \\
          1 & x=1\\
          0 & x=0
        \end{cases}.
      \]
    \item[(vii)]
      If $W$ is of type $I_2(n)$, then \[
        {\sf maxw}_0(W, x) :=
        \begin{cases}
          n & x = 2 \\
          x & x < 2
        \end{cases}.
      \]
  \end{itemize}
\end{proposition}

\begin{proof}
  Figure~\ref{fig:coxG} contains the Coxeter diagrams and the length of the longest element for each of the finite Coxeter groups in this proof. Let $I \subset [n]$ with $|I| = x$. 

  \noindent \textit{(i)}
    When $W$ is of type $A_n$, then ${\mathcal G}_I$ decomposes into connected components
    $C^{(1)},\ldots,C^{(z)}$ of type $A_{k_1},\ldots,A_{k_z}$ respectively,
    with $k_1,\ldots,k_z > 0$, $z \geq 1$, and $k_1 + \cdots + k_z = x$.
    By \eqref{eqn:thelength}, 
    \begin{align*}
    \ell(w_0(W_I)) = \ell(w_0(A_{k_1})) + \cdots + \ell(w_0(A_{k_z})) 
    & = {k_1 + 1 \choose 2}+ \cdots + {k_z + 1 \choose 2} & \\
    & = \frac{k_1^2 + \cdots k_z^2 + k_1 + \cdots + k_z}{2} & \\
    & \leq \frac{(k_1 + \cdots + k_z + 1)(k_1 + \cdots + k_z)}{2} & \\
    & = {k_1 + \cdots + k_z + 1 \choose 2} = {x + 1 \choose 2}. &
    \end{align*}
    This upper bound is realized when $z=1$ and $k_1 = x$. Thus ${\sf maxw}_0(W, x)={x + 1 \choose 2}$.

  \noindent \textit{(ii)}
    When $W$ is of type $B_n$ or $C_n$, then ${\mathcal G}_I$ decomposes into connected components of type
    $A_{k_1},\ldots,A_{k_z}$ and $B_m$, with $k_1,\ldots,k_z > 0$, $z \geq 0$,
    $m \geq 0$ and $k_1 + \cdots + k_z + m = x$.
    Note that when $m=0$ (respectively, $z=0$) we take this to mean that there
    is no connected component of ${\mathcal G}_I$ of type $B_{s}$
    (respectively, of type $A_{s}$) for any natural number $s$.
    Thus, \eqref{eqn:thelength} implies 
    \begin{align*}
    \ell(w_0(W_I)) & = \ell(w_0(A_{k_1})) + \cdots + \ell(w_0(A_{k_z})) + \ell(w_0(B_m)) & \\
    & = {k_1 + 1 \choose 2} + \cdots + {k_z + 1 \choose 2} + m^2 & \\
    & \leq {k_1 + \cdots + k_z + 1 \choose 2} + m^2 & \\
    & =  {x - m + 1 \choose 2} + m^2 & \\
    & = \frac{3}{2} m^2 - \frac{2x+1}{2}m + \frac{x^2 + x}{2} =: f_1(m). &
    \end{align*}
    Then $m$ is an integer value in the closed interval $[0,x]$. The function $f_1$ is convex on the closed interval $[0,x]$ and hence $f_1$ achieves its maximum at one of the endpoints. For $x\!\geq\!0$, $f_1(0) = {x + 1 \choose 2} \leq x^2 = f_1(x)$.
    This upper bound is realized when $z\!=\!0$, $m\!=\!x$. It follows that ${\sf maxw}_0(W, x)\!=\!x^2$.

  \noindent \textit{(iii)}
    When $W$ is of type $D_n$, then ${\mathcal G}_I$ decomposes into connected components of type
    $A_{k_1},\ldots,A_{k_z}$ and $D_m$, with $k_1,\ldots,k_z > 0$, $z \geq 0$,
    $m \geq 0$ with $m \neq 1,2,3$, and $k_1 + \cdots + k_z + m = x$.
    As in \textit{(ii)}, if $m=0$ (respectively, $z=0$), then we take this to
    mean there are no connected components in ${\mathcal G}_I$ of type $D_s$
    (respectively, of type $A_s$) for any natural number $s$.
    Now, \eqref{eqn:thelength} implies 
    \begin{align*}
    \ell(w_0(W_I)) & = \ell(w_0(A_{k_1})) + \cdots + \ell(w_0(A_{k_z})) + \ell(w_0(D_m)) & \\
    & = {k_1 + 1 \choose 2} + \cdots + {k_z + 1 \choose 2} + m^2 - m & \\
    & \leq {k_1 + \cdots + k_z + 1 \choose 2} + m^2 - m & \\
    & =  {x - m + 1 \choose 2} + m^2 - m & \\
    & = \frac{3}{2} m^2 - \frac{2x+3}{2}m + \frac{x^2 + x}{2} =: f_2(m). &
    \end{align*}
    Then $m$ is an integer value in the interval $[0,x]$. The function $f_2$ is convex on the closed interval $[0,x]$ and hence $f_2$ achieves its maximum at one of the endpoints. If $x > 3$, then $f_2(0) = {x + 1 \choose 2} \leq x^2- x = f_2(x)$. The upper bound is realized when $z=0$ and $m=x$ for $x > 3$. Hence ${\sf maxw}_0(W, x)\!=\!x^2\!-\!x$ for $x\!>\!3$. 

    If $x \leq 3$, then $f_2(0)\!=\!{x + 1 \choose 2}\!\geq\!x^2- x = f_2(x)$. This upper bound is realized when $z=1$, $m=0$, and $k_1 = x$. Hence ${\sf maxw}_0(W, x) = {x + 1 \choose 2}$ for $x \leq 3$. 

  \noindent \textit{(iv)-(vii)}
    Each of these cases can be trivially checked via the enumeration of all
    induced subdiagrams of a fixed size.
\end{proof}

\subsection{Spherical implies proper}
We show that sphericality, both in the geometric and Coxeter sense,
implies properness.
This allows the proofs of Theorem \ref{thm:main2}, \ref{thm:main3} to be reduced
to Theorem \ref{thm:main1}.
The following is a generalization of \cite[Proposition 3.1]{BHY20}.
\begin{proposition}
\label{prop:geosphericalimpliesproper}
  Let $G$ be a rank $r$ simple group with Weyl group $W$.
  If $X_w \subseteq G/B$ is $L_I$-spherical for $I \subseteq J(w)$,
  then $w$ is proper.
\end{proposition}

\begin{proof}
  If $X_w$ is $L_I$-spherical, then $X_w$ is
  $L_{J(w)}$-spherical~\cite[Proposition 2.13]{HY20}.
  By definition, $X_w$ is $L_{J(w)}$-spherical implies that there is a Borel
  subgroup $K \subset L_{J(w)}$ with an open dense orbit ${\mathcal O}$ in $X_w$.
  For $x \in {\mathcal O}$, let $K_x$ be the isotropy group of $x$.
  By \cite[Proposition 1.11]{Brion}, ${\mathcal O} = K \cdot x$ is a smooth,
  closed subvariety of $X_w$ of dimension $\dim{K} - \dim{K_x}$.
  Thus
  \begin{equation}\label{eq:dimUBSpherical}
    \dim{X_w}
    = \dim{{\mathcal O}}
    = \dim{K}-\dim{K_x}
    \leq \dim{K},
  \end{equation}
  where the first equality follows since ${\mathcal O}$ is dense in $X_w$.
  All Borel subgroups of a connected algebraic group are conjugate
  ~\cite[\S 11.1]{Borel}.
  Hence
  \begin{equation}\label{eq:dimUBSpherical2}
    \dim{K}
    = \dim{B_{J(w)}}
    = \dim{T} + \dim{U_{J(w)}},
  \end{equation}
  where the final equality follows from \cite[\S 11.1]{Borel}.
  Finally, we use the fact that $\dim{U_{J(w)}}$ equals the number of positive
  roots in the root system of $L_{J(w)}$, which is in turn equal to
  $\ell(w_0(W_{J(w)}))$~\cite[\S 1.7]{H90}.
  Combining this with \eqref{eq:dimUBSpherical} and \eqref{eq:dimUBSpherical2}
  yields \[
    \ell(w) 
    = \dim{X_w} % TODO: clarify why this is true.
    \leq \dim{T} + \dim{U_{J(w)}}
    = r + \ell(w_0(W_{J(w)}))
    \leq r + {\sf maxw}_0(W,d(w)).
  \]
  We conclude that $w$ is proper.
\end{proof}

\begin{proposition}
\label{prop:coxsphericalimpliesproper}
  Let $W$ be a finite Coxeter group with $w\in W$ and $I \subseteq J(w)$.
  If $w$ is $I$-spherical, then $w$ is proper.
\end{proposition}

\begin{proof}
  Let ${\mathcal C}^{(1)},\ldots,{\mathcal C}^{(m)}$ be the connected components
  of ${\mathcal G}_{J(w)}$, and $W^{(z)}$ the parabolic subgroup of $W$ with
  Coxeter diagram ${\mathcal C}^{(z)}$.
  By ~\cite[Proposition 2.12]{HY20}, if $w$ is $I$-spherical, then $w$ is
  $J(w)$-spherical. Thus there must be an $R \in {\text{Red}}(w)$ satisfying \eqref{S.1} and \eqref{S.2}.
  This implies, via \eqref{S.1}, that at most $n - d(w)$ factors of $R$ must be of
  the form $s_j$ for $j \notin J(w)$. Which implies at least $\ell(w) - (n-d(w))$ factors of $R$ are of the form $s_j$ for $j \in J(w)$.
  Hence, by \eqref{S.2} and \eqref{eqn:thelength},
  \begin{align*}
    \ell(w) - (n-d(w))
    & \leq  \displaystyle \sum_{z=1}^m \left( \ell(w_0({W^{(z)}})) + \#\text{vertices}({\mathcal C}^{(z)}) \right) & \\
    & = d(w) + \sum_{z=1}^m \ell(w_0({W^{(z)}})) &  \\
    & = d(w) + \ell(w_0(W_{J(w)})) & \\
    & \leq d(w) + {\sf maxw}_0(W,d(w)). &
  \end{align*}
  It follows that $w$ is proper.
\end{proof}

We are now able to prove both Theorem~\ref{thm:main2} and Theorem~\ref{thm:main3}, assuming Theorem~\ref{thm:main1}.

\noindent \textit{Proof of Theorem~\ref{thm:main2}:} If $X_w$ is $L_{J(w)}$-spherical, then $w$ is proper by Proposition~\ref{prop:geosphericalimpliesproper}. Hence
\begin{equation}
\label{eq:squeeze}
{\rm Pr}[X_w \subseteq G/B \text{\ is $L_{J(w)}$-spherical}] \leq {\rm Pr}[w \text{\ is proper}]
\end{equation}
when $w$ is sampled from $W$ uniformly at random.
Theorem~\ref{thm:main1} implies that ${\rm Pr}[w \text{\ is proper}]\longrightarrow 0$ as $n\longrightarrow\infty$. Thus our desired result follows by \eqref{eq:squeeze} and the squeeze theorem.

\noindent \textit{Proof of Theorem~\ref{thm:main3}:} This follows by an identical argument after applying Proposition~\ref{prop:coxsphericalimpliesproper}.

\section{Concentration bounds}
\label{sec:prob}

\subsection{Concentration Bounds}
In this section we compute some concentration bounds that will be useful for
bounding the number of proper elements of a Coxeter group in Section
\ref{section:5}.
Lemmas \ref{lemma:chernoff} and \ref{lemma:1} will be used to bound the number of
elements with left descents deviating far from the average.
Lemmas \ref{lemma:4} and \ref{lemma:2} will be used to bound the number of
elements with length deviating far from the average.
Finally, Lemma \ref{lemma:3} will be used to bound the number of
elements satisfying a certain inequality involving length and left descents.
See Chapter 3 of~\cite{RP23} for some of the terminology
and basic properties of conditional expectations and martingales that are used in this section.

\begin{theorem}\label{theorem:chernoff-dp}(Chernoff Bound; from Theorem 1.1 of \cite{DP09})
Let $X_1,\ldots,X_n$ be (mutually) independently distributed random variables in the range $[0,1]$, and let
$X := \sum_i X_i$.
Then for any $\gD\in(0,1)$,
\begin{equation}
\label{eq:chernoff-dp}
    \Pr{X<(1-\gD)\E{X}}\leq e^{-\gD^2\E{X}/2}.
\end{equation}
\end{theorem}

\begin{lemma}\label{lemma:chernoff}
  Let $X_1,\ldots,X_m$ be mutually independently distributed random variables in the range $\{0,1\}$.
  Let $X:=\sum_i X_i$.
  Then for any $\gd>0$, we have both \[
    \Pr{X-\E{X}>\gd m}\leq e^{-\gd^2m/2}
    \quad\text{and}\quad
    \Pr{\E{X}-X>\gd m}\leq e^{-\gd^2m/2}.
  \]
\end{lemma}
\begin{proof}
  If $\E{X}=0$, then $X=0$, in which case $\Pr{\E{X}-X>\gd m}=\Pr{\gd m<0}=0$. If $\gd m\geq\E{X}$, then $\Pr{\E{X}-X>\gd m}\leq\Pr{\E{X}-X>\E{X}}=\Pr{X<0}=0$. In the case when $\gd m < \E{X}$,
  \begin{align*}
    \Pr{\E{X}-X>\gd m}
    &=\Pr{X<(1-\nicefrac{\gd m}{\E{X}})\E{X}} \\
    &\leq\Exp{-\left(\frac{\gd m}{\E{X}}\right)^2\cdot\frac{\E{X}}2} &\text{by \cref{theorem:chernoff-dp}} \\
    &=\Exp{\frac{-\gd^2 m}2\cdot\frac{m}{\E{X}}} \\
    &\leq e^{-\gd^2 m/2}& \\
  \end{align*}
  where the last inequality from the fact that $m\geq\E{X}$.
  To show the former bound, set $X_i':=1-X_i$ and $X':=\sum_i X_i'$.
  Then $X-\E{X}=\E{X'}-X'$, and a similar analysis to the one above can be performed on $X'$.
\end{proof}

\begin{lemma}\label{lemma:1}
  Let $\Set{I_\ga}_{\ga\in\gC}$ be a set of identically distributed
  random variables in the range $\{0,1\}$, and let $\gC_1,\ldots,\gC_k\subseteq\gC$ be such
  that:
  \begin{enumerate}[label=(\roman*)]
    \item $\gC_i\cap\gC_j=\emptyset$, for all $i\neq j$;
    \item $\sum_i|\gC_i|=n-1$;
    \item $\Floor{(n-1)/k}\leq|\gC_i|\leq\Floor{n/k}$;
    \item for each $i \in [k]$, the elements of $\Set{I_a}_{a\in\Gamma_i}$ are mutually independent.
  \end{enumerate}
  Let $S_i=\sum_{a\in\gC_i}I_a$ and $S=\sum_i S_i$.
  Then for any $\gd>0$, we have both \[
    \Pr{S-\E{S}>\gd n}\leq ke^{\nicefrac{-\gd^2(n-2)}{2k}}
    \quad\text{and}\quad
    \Pr{\E{S}-S>\gd n}\leq ke^{\nicefrac{-\gd^2(n-2)}{2k}}.
  \]
\end{lemma}
\begin{proof}
  By Lemma \ref{lemma:chernoff} we have
  \begin{equation}
  \label{eq:chernoff2}
  \Pr{S_i-\E{S_i}>\gd|\gC_i|}\leq e^{-\gd^2|\gC_i|/2}
  \end{equation}
  for each $i \in [k]$. Thus,
  \begin{align*}
    \Pr{S-\E{S}\geq\gd n}
    &=\Pr{\sum_i (S_i-\E{S_i})\geq\gd n} \\
    &\leq \Pr{\bigcup_{i}\Set{S_i-\E{S_i}\geq\gd n/k}} \\
    &\leq \sum_i\Pr{S_i-\E{S_i}\geq\gd n/k} &\text{by the union bound} \\
    &\leq \sum_i\Pr{S_i-\E{S_i}\geq\gd|\gC_i|} &\text{since }|\gC_i|\leq\Floor{n/k}\leq n/k \\
    &\leq \sum_i e^{-\gd^2|\gC_i|/2} &\text{by}~\eqref{eq:chernoff2}\\
    &\leq ke^{\nicefrac{-\gd^2(n-2)}{2k}}&\text{since }(n-2)/k \leq \Floor{(n-1)/k} \leq |\gC_i| \\
  \end{align*}
  Similarly, $\Pr{\E{S}-S\geq\gd n}\leq ke^{\nicefrac{-\gd^2(n-2)}{2k}}$ since the bound from Lemma \ref{lemma:chernoff} is two-tailed.
\end{proof}

\begin{theorem}\label{thm:azuma}(Azuma-Hoeffding Inequality; Theorem 5.8 of \cite{DP09})
Let $X_0, X_1, \ldots$ be a martingale and let $b_1,b_2,\ldots$ be a sequence of non-negative constants such that $|X_i-X_{i-1}|\leq b_i$ for each $i\geq1$.
Then, \[
    \Pr{X_n > X_0 + t} \leq \Exp{-\frac{t^2}{2\gb_n}}
\quad\text{and}\quad
\Pr{X_n < X_0 - t} \leq \Exp{-\frac{t^2}{2\gb_n}}
\]
where $\gb_n = \sum_{i=1}^n b_i^2$.
\end{theorem}

We use the following elementary lemma; we omit the proof.
\begin{lemma}\label{lemma:4}
  Let $0\leq x_0,\ldots,x_m\leq1$ such that $x_0 + \cdots + x_m=1$ and
  $x_k=x_{m-k}$ for $k=0,\ldots,m$.
  Then we have $\sum_{k=0}^{m}kx_k = m/2$.
\end{lemma}

\begin{lemma}\label{lemma:2}
  Suppose $Q_0,Q_1,Q_2,\ldots$ is a sequence of random variables such that
  \begin{enumerate}[label=(\roman*)]
    \item\label{cond:l1} $Q_0=0$,
    \item\label{cond:l2} $\sPrg{Q_{t+1}}{Q_t,\ldots,Q_0}=\sPrg{Q_{t+1}}{Q_t}$,
    \item\label{cond:l3} $\sPrg{Q_{t+1}=Q_{t}+a}{Q_t}=\sPrg{Q_{t+1}=Q_{t}+(c_t-a)}{Q_t}$ for all $a=0,\ldots,c_t$, and
    \item\label{cond:l4} $\sPrg{Q_{t+1}=Q_{t}+a}{Q_t}=0$ for all $a\not\in\Set{0,\ldots,c_t}$,
  \end{enumerate}
  where $c_0,c_1,c_2,\ldots$ is a sequence of non-negative integers.
  Let $\gc_n=\sum_{i=0}^{n-1}c_i^2$.
  Then for all $\e>0$, \[
    \Pr{Q_n-\E{Q_n}\geq\e}\leq e^{-2\e^2/\gc_n}
    \quad\text{and}\quad
    \Pr{\E{Q_n}-Q_n\geq\e}\leq e^{-2\e^2/\gc_n}.
  \]
\end{lemma}
\begin{proof}
  Firstly,
  \begin{equation}\label{eq:A5}
    \sum_{a=0}^{c_t}a\cdot\sPrg{Q_{t+1}=Q_t+a}{Q_t}
    =c_t/2
  \end{equation}
  by Lemma \ref{lemma:4}.
  From this,
  \begin{equation}\label{eq:A3}
    \sEg{Q_{t+1}}{Q_t}
    =Q_t+\sum_{a=0}^{c_t}a\cdot\sPrg{Q_{t+1}=Q_t+a}{Q_t}
    =Q_t+c_t/2.
  \end{equation}
  Thus, by the law of total expectation,
  \begin{equation}\label{eq:A4}
    \E{Q_{t+1}}=\E{\sEg{Q_{t+1}}{Q_t}}=\E{Q_t}+c_t/2.
  \end{equation}
  Letting $Z_t:=Q_t-\E{Q_t}$ with $Z_0=Q_0$, then $(Z_t)_{t\geq0}$ is a martingale.
  
  Also, $\E{Z_t}=0$ for each $t$, using \ref{cond:l1} for $t=0$.
  Since $Q_t-Q_{t-1}\in\Set{0,\ldots,c_{t-1}}$ and
  $\E{Q_t}-\E{Q_{t-1}}=c_{t-1}/2$,
  \begin{align*}
    |Z_t-Z_{t-1}|
    &=|(Q_t-Q_{t-1})+(\E{Q_{t-1}}-\E{Q_t})|
    \leq c_{t-1}/2. \\
  \end{align*}
  Applying Theorem \ref{thm:azuma} to $(Z_t)_{t\geq0}$ with the bounded differences $b_t := c_{t-1}/2$ for $t\geq1$ gives that
  $\Pr{Z_n\geq \e}$ and $\Pr{Z_n\leq -\e}$ are both bounded above by $e^{-2\e^2/\gc_n}$.
  By the definition of $Z_n$, we have that $\Pr{Q_n-\E{Q_n}\geq\e}=\Pr{Z_n\geq\e}$ and
  $\Pr{\E{Q_n}-Q_n\geq\e}=\Pr{Z_n\leq -\e}$.
\end{proof}

\begin{lemma}\label{lemma:3}
  Let $\Set{X_i}_{i\in\gC}$ be a set of random variables.
  Let $\Set{f_i}_{i\in\gC}$ be increasing functions over the non-negative reals.
  Let $(\gC^+,\gC^-)$ be a partition of $\gC$.
  For any set $\Set{\e_i}_{i\in\gC}$ of non-negative reals and any real number $r$ that satisfy
  \begin{enumerate}[label=(\roman*)]
    \item $\e_a\leq\E{X_a}$ for all $a\in\gC^+$,
    \item $\e_b\geq-\E{X_b}$ for all $b\in\gC^-$, and
    \item $\sum_{a\in\gC^+}f_a(\E{X_a}-\e_a)-\sum_{b\in\gC^-}f_b(\E{X_b}+\e_b) > r$,
  \end{enumerate}
  we have \[
    \Pr{\sum_{a\in\gC^+}f_a(X_a)-\sum_{b\in\gC^-}f_b(X_b)\leq r}
    \leq \sum_{a\in\gC^+}\Pr{\E{X_a}-X_a\geq\e_a}+\sum_{b\in\gC^-}\Pr{X_b-\E{X_b}\geq\e_b}.
  \]
\end{lemma}
\begin{proof}
  We have
  \[
  \begin{split}
  \Pr{\sum_{a\in\gC^+}f_a(X_a)-\sum_{b\in\gC^-}f_b(X_b) > r} \qquad\qquad\qquad\qquad\qquad\qquad\qquad\qquad\qquad\qquad\qquad\qquad \\
  \begin{aligned}
    &\geq\Pr{\left(\bigcap_{a\in\gC^+}\Set{X_a>\E{X_a}-\e_a}\right)\cap
             \left(\bigcap_{b\in\gC^-}\Set{X_b<\E{X_b}+\e_b}\right)} \\
    &=\Pr{\left(\bigcap_{a\in\gC^+}\Set{\E{X_a}-X_a<\e_a}\right)\cap
             \left(\bigcap_{b\in\gC^-}\Set{X_b-\E{X_b}<\e_b}\right)} \\
    &=1-\Pr{\left(\bigcup_{a\in\gC^+}\Set{\E{X_a}-X_a\geq\e_a}\right)\cup
             \left(\bigcup_{b\in\gC^-}\Set{X_b-\E{X_b}\geq\e_b}\right)} \\
    &\geq1-\sum_{a\in\gC^+}\Pr{\E{X_a}-X_a\geq\e_a}
          -\sum_{b\in\gC^-}\Pr{X_b-\E{X_b}\geq\e_b}.
  \end{aligned}
  \end{split}
  \]
  The first step uses the fact that $\Set{f_i}_{i\in\gC}$ is a set of increasing functions;
  the last step uses the union bound.
  Taking the probability complement of this yields our desired result.
\end{proof}

%%%%%%%%%%%%%%%%%%%%%%%%% Start types specific proofs %%%%%%%%%%%%%%%%%%%%%%%%%

\section{The infinite families}
\label{sec:combmodels}

The finite, irreducible Coxeter groups of types $A_{n-1}, B_n$, and $D_n$ have
combinatorial interpretations in terms of permutations.
We follow the notation and repeatedly use results from \cite{BB06} in this section.
The type $A_{n-1}$ Coxeter group is isomorphic to the symmetric group $S_n$
under the map that sends each $s_i\in S$ to the simple transposition $(i, i+1)$.
Thus the order of $A_{n-1}$ is $n!$.
For any bijection $w$ on a subset $S \subset \mathbb{Z}$, let the number of inversions be
\[
  \mathrm{inv}(w):=\#\{(i,j)\mid w(i) > w(j), 1\leq i<j\leq n\}.
\]
The length of an element of type $A_{n-1}$ is
equal to the number of inversions of $w$ (thinking of $w$ as an element of $S_n$);
that is, $\ell_A(w) = \mathrm{inv}(w)$.

The type $B_{n}$ Coxeter group is isomorphic to the group of \textit{signed} permutations,
$S_n^B$, which is the collection of bijections $w$ on the set
$[\pm n]=\Set{-n,\cdots,-1,1,\cdots,n}$ with the property that $w(-i)=-w(i)$
for every $i\in[\pm n]$, under the binary operation of function composition.
The order of this group is $2^nn!$.
The length of $w \in S_n^B$ can be expressed as \[
  \ell_B(w)=\mathrm{inv}_B(w):=\mathrm{inv}(w) +\mathrm{nsp}(w) +\mathrm{neg}(w)
\]
where
\begin{align*}
\mathrm{nsp}(w) & :=\#\{(i, j)\mid w(i)+w(j)<0, 1\leq i<j\leq n\}, \\
\mathrm{neg}(w) & :=\#\{i\mid w(i)<0, 1\leq i\leq n\}.
\end{align*}
The type $D_{n}$ Coxeter group is isomorphic to the subgroup $S_n^D \subset S_n^B$ containing all $w$ such that $\mathrm{neg}(w)$ even.
The order of this group is $2^{n-1}n!$.
The length of $w \in S_n^D$ can be expressed as \[
  \ell_D(w) = \mathrm{inv}_D(w):=
  \mathrm{inv}_B(w)-\mathrm{neg}(w)
  =\mathrm{inv}(w)+\mathrm{nsp}(w).
\]
For any bijection $w$ on a subset $S \subset \mathbb{Z}$, let
\[
  \mathrm{des}(w):=\#\{i \mid\iv{w}(i) > \iv{w}(i+1), 1\leq i < n\}.
\]
The number of left descents (Definition \ref{def:left-descents}) of an
element $w$ in a Coxeter group of type $A_{n-1}$, $B_n$, and
$D_n$ can be computed, respectively, as
\begin{itemize}
  \item[(i)] $d_A(w)=\mathrm{des}(w)$, %\quad\cite[p. 17, p. 21 (1.28)]{BB06},
  \item[(ii)] $d_B(w)=\mathrm{des}_B(w):=\mathrm{des}(w)+ \mathbbm1[0 > \iv w(1)]$, %\quad\cite[p. 248 (8.8)]{BB06},
  \item[(iii)] $d_D(w)=\mathrm{des}_D(w):=\mathrm{des}(w)+ \mathbbm1[\iv w(-2)>\iv w(1)]$, %\quad\cite[p. 254 (8.23), ??]{BB06}.
\end{itemize}
where $\mathbbm1[P]$ is the \emph{Iverson bracket} defined as
\[
\mathbbm1[P] = \begin{cases} 1 & P\text{ is true}; \\ 0 & \text{otherwise}. \\ \end{cases}
\]
In the analysis that follows, the \emph{asymptotic substitution} $f(n)=O(g(n))$
means $f(n)\leq c\cdot g(n)$ for all $n\geq N$ where $N$ and $c$ are some
absolute positive constants with respect to $n$.
We write $f(n)=\gO(g(n))$ to mean $g(n)=O(f(n))$.
When either $O(\cdot)$ or $\gO(\cdot)$ appears as part of an inequality,
for all possible asymptotic substitutions on the left hand side, there must exist
asymptotic substitutions on the right hand side such that the inequality holds.
The general technique will be to generate a uniformly at random Coxeter element
and bound its Coxeter related statistics with high probability.

The \textit{one-line notation} of a permutation $w\in S_n$ is the string
$w(1)\cdots w(n)$.
For an element $w$ in either $S_n^B$ or $S_n^D$, its one-line notation is
$(w(-n),\ldots,w(-1),w(+1),\ldots,w(+n))$, and has indices in $[\pm n]$.

\section{Limiting behavior for the proportion of proper elements}\label{section:5}
In \cite{BHY20}, it was shown that the proportion of proper \textit{permutations} is
$O(\iv n)$ and hence is asymptotically zero.
In this section, we show that every group of type $A_{n-1}$ has an
$e^{-\gO(n)}$ proportion of proper elements.
We then give similar asymptotic results for types $B_n$ and $D_n$.

\subsection{Type $A_{n-1}$} We begin with type $A_{n-1}$ Coxeter groups.

\begin{proposition}\label{prop:an}
The number of proper elements in the type $A_{n-1}$ Coxeter group is at most
$n!\cdot e^{-\gO(n)}$.
\end{proposition}
\begin{proof}
  For $1\leq i<j\leq n$, let $X_{i,j}^{(n)}\colon S_n\to\Set{0,1}$ be the function
  that maps $w\in S_n$ to $1$ if $\iv w(i)>\iv w(j)$ and $0$ otherwise.
  Then \[
      {\rm inv}(w)=\sum_{i=1}^n\sum_{j=i+1}^nX_{i,j}^{(n)}(\iv w)
      \quad \text{ and } \quad
      {\rm des}(w)=\sum_{i=1}^{n-1}X_{i,i+1}^{(n)}(w)
  \]
  for $w\in S_n$.
  Further, $\mathbb{E}_w[{X_{i,j}^{(n)}(w)}]=1/2$ via symmetry.
  Thus $\mathbb{E}_w[{{\rm inv}(w)}]=\binom{n}2/2$
  and $\mathbb{E}_w[{\rm des}(w)]=(n-1)/2$ by the linearity of expectation.
  We also have that the random variables $X_{i,j}^{(n)}$ and $X_{i',j'}^{(n)}$ are independent when $\Set{i,j}\cap\Set{i',j'}=\emptyset$.

  Let $\gC=\Set{(i,j)\mid 1\leq i<j\leq n}$
  with $\gC_1=\{(i,i+1)\mid 1\leq i<n, i\equiv 1\Mod2\}$
  and $\gC_2=\{(i,i+1)\mid 1\leq i<n, i\equiv 0\Mod2\}$.
  Notice $\gC_1$ and $\gC_2$ are disjoint subsets of $\gC$,
  and that both $\{X_{a}^{n}\}_{a\in\gC_1}$
  and $\{X_{b}^{n}\}_{b\in\gC_2}$
  are sets of independent random variables.
  Further, $|\gC_1|+|\gC_2|=n-1$ with $|\gC_1|=\Ceil{(n-1)/2}$
  and $|\gC_2|=\Floor{(n-1)/2}$.
  Also, \[
      \sum_{a\in\gC_1}X_{a}^{(n)}(w)+\sum_{b\in\gC_2}X_b^{(n)}(w)
      =\sum_{i=1}^{n-1}X_{i,i+1}^{(n)}(w)
      ={\rm des}(w)
  \]
  for $w\in S_n$.
  Thus Lemma \ref{lemma:1} with $\gd=\e/n$ gives us
  \begin{equation}\label{eq:A1}
    \Pr{{\rm des}(w)-\E{{\rm des}(w)}\geq\e}
    \leq 2\Exp{-\frac{\e^2(n-2)}{4n^2}}
  \end{equation}
  for any $\e>0$ when $w$ is sampled uniformly at random from $S_n$.

  Next, let $\gp_0,\gp_1,\gp_2,\ldots$ be a sequence of random permutations,
  evolving such that given $\gp_{t-1}\in S_{t-1}$, we create $\gp_{t}$ by inserting
  $t$ at a (uniformly) random index in $\gp_{t-1}$, such that
  ${\mathrm{Pr}[\gp_t(i)=t\mid\gp_{t-1}]}=1/t$
  for each $i\in [t]$.
  Let $Q_t:=\mathrm{inv}(\gp_t)$.
  From this, we see that $\sPrg{Q_{t+1}=Q_{t}+a}{Q_t}=1/(t+1)$ for $a=0,\ldots,t$.
  Since the distribution is uniform, it is symmetric about $a=t/2$.
  Then setting $c_t=t$, Lemma \ref{lemma:2} gives
  \begin{equation}\label{eq:preA2}
    \Pr{\E{Q_n}-Q_n\geq\e}
    \leq \Exp{-\frac{12\e^2}{n(n-1)(2n-1)}}
  \end{equation}
  for any $\e>0$.

  \noindent Notice that ${\rm inv}(w_1)$ and ${\rm inv}(\iv{w_2})$ are identically distributed when
  $w_1$ and $w_2$ are independently sampled uniformly at random from $S_n$ since \[
    {\rm inv}(w) = \ell_A(w) = \ell_A(\iv w) = {\rm inv}(\iv w).
  \]
  If $\gp_t$ is the evolved permutation after $t$ insertions,
  then $\gp_{t-1}$ is completely determined by $\gp_t$.
  Further, there is only one insertion of $t$ that evolves $\gp_{t-1}$ into
  $\gp_t$.
  The conclusion, via induction, is that the probability of seeing $\gp_t$ after $t$ insertions
  is $1/t!$. The above implies that $Q_n$ has an identical distribution to
  ${\rm inv}(\iv w)$, and hence ${\rm inv}(w)$, when $w$ is sampled uniformly at random from $S_n$.
  This, combined with \eqref{eq:preA2}, implies
  \begin{equation}\label{eq:A2}
    \mathrm{Pr}_w{[\E{\rm inv(w)}-{\rm inv(w)}\geq\e]}
    \leq \Exp{-\frac{12\e^2}{n(n-1)(2n-1)}}
  \end{equation}
  for any $\e>0$.

  Finally, take $f_1$ and $f_2$ to be functions over the reals with $f_1(x)=x$
  and $f_2(x)=(x+1)x/2$.
  Also, let $r=n-1$, $\e_1=n^2/16$, and $\e_2=n/16$.
  Using Lemma \ref{lemma:3}, for sufficiently large $n$, with
  \eqref{eq:A1} and \eqref{eq:A2} gives
  \begin{align*}
    \#\Set{w\in S_n\mid w\text{ is proper}}
    &=|S_n|\cdot\mathop{\mathrm{Pr}}\limits_{w\gets S_n}[w \text{ is proper}] \\
    &=n!\cdot\mathop{\mathrm{Pr}}\limits_{w\gets S_n}\left[{\rm inv}(w)\leq(n-1)+\binom{{\rm des}(w)+1}2\right] \\
    &\leq n!\cdot\left[2\Exp{-\frac{n-2}{256}} + \Exp{-\frac{3n^3}{64(n-1)(2n-1)}} \right] \\
    &\leq n!\cdot e^{-\gO(n)}.
  \end{align*}\end{proof}

\begin{corollary}
\label{cor:asymptoticA}
  The proportion of proper elements in the Coxeter group of type $A_{n-1}$
  vanishes as $n$ goes to infinity.
\end{corollary}
\begin{proof}
  We know the type $A_{n-1}$ Coxeter group has group structure $S_n$.
  Also, $|S_n|=n!$.
  Thus, the claim follows from Proposition \ref{prop:an}
  since the proportion of proper elements is \[
    \mathop{\mathrm{Pr}}\limits_{w\gets S_n}[w \text{ is proper}]
    \leq e^{-\gO(n)}
  \]
  which tends to zero as $n$ goes to infinity.
\end{proof}

\subsection{Types $B_n$ and $D_n$}
Next, we give a bound involving type $A_{n-1}$ left descents, type $B_n$ elements,
and type $D_n$ inversions.
We do this to consolidate our bounds in types $B_n$ and $D_n$.
Types $B_n$ and $D_n$ have asymptotically the same number of elements, so their
analysis is similar.
In providing the following, we upper bound a large set and then later show that
this set is a superset of both the set of type $B_n$ and type $D_n$ proper
elements.

\begin{proposition}\label{prop:snb}
  We have that \[
    \#\Set{w\in S_n^B\mid \mathrm{inv}_D(w)\leq n+(\mathrm{des}(w)+1)^2}
    \leq 2^nn!\cdot e^{-\gO(n)}.
  \]
\end{proposition}
\begin{proof}
  This proof is similar to the proof of Proposition \ref{prop:an}.
  For $1\leq i<j\leq n$, let $X_{i,j}^{(n)}\colon S_n^B\to\Set{0,1}$ be the
  function that maps $w\in S_n^B$ to $1$ if $\iv w(i)>\iv w(j)$ and $0$ otherwise.
  Similarly, let $Y_{i,j}^{(n)}\colon S_n^B\to\Set{0,1}$ be the function that
  maps $w\in S_n^B$ to $1$ if $\iv w(-i)>\iv w(j)$ and $0$ otherwise.

  Then \[
    {\rm inv}_D(w)=
        \sum_{i=1}^n\sum_{j=i+1}^nX_{i,j}^{(n)}(w)+Y_{i,j}^{(n)}(w)
  \]
  and 
  \[
    {\rm des}(w)=
        \sum_{i=1}^{n-1}X_{i,i+1}^{(n)}(w),
  \]
  for $w\in S_n^B$ since by \cite[p. 253 (8.18)]{BB06}
  \begin{align*}
    {\rm inv}_D(w)
    &=\mathrm{inv}_B(w)-\mathrm{neg}(w) \\
    &=\mathrm{inv}(w)+\#\{1\leq i<j\leq n\mid w(-i)>w(j)\} &\\
    &={\rm inv}_D(\iv w).
  \end{align*}

  For $i, j\in[n]$ with $i<j$, let $\iota_{i,j}: S_n^B \to S_n^B$ be the
  bijective involution that swaps the values at indices $i$ and $j$ in the one-line notation of $w\in S_n^B$.
  Then both $X_{i,j}^{(n)}(w) = 1 - X_{i,j}^{(n)}(\gi_{i,j}(w))$ and
  $Y_{i,j}^{(n)}(w) = 1 - Y_{i,j}^{(n)}(\gi_{i,j}(w))$.
  Hence $\mathbb{E}{X_{i,j}^{(n)}}=\mathbb{E}{Y_{i,j}^{(n)}}=1/2$ by symmetry.
  Thus, $\mathbb{E}[{\rm inv}_D(w)]=\binom{n}2$ and $\mathbb{E}[{\rm inv}_D(w)]=(n-1)/2$ by the
  linearity of expectation. Further, $X_{i,j}^{(n)}$ and $X_{i',j'}^{(n)}$ are independent if
  $\Set{i,j}\cap\Set{i',j'}=\emptyset$.

  Let $\gC=\Set{(i,j)\mid 1\leq i<j\leq n}$ with
  $\gC_1=\{(i,i+1)\mid 1\leq i<n, i\equiv 1\Mod2\}$
  and
  $\gC_2=\{(i,i+1)\mid 1\leq i<n, i\equiv 0\Mod2\}$.
  Notice $\gC_1$ and $\gC_2$ are disjoint subsets of $\gC$, and that both
  $\{X_{a}^{n}\}_{a\in\gC_1}$ and $\{X_{b}^{n}\}_{b\in\gC_2}$ are sets of
  independent random variables.
  Further, $|\gC_1|+|\gC_2|=n-1$ with $|\gC_1|=\Ceil{(n-1)/2}$ and
  $|\gC_2|=\Floor{(n-1)/2}$.
  Since for $w\in S_n^B$, $\sum_{a\in\gC_1}X_{a}^{(n)}(w)+\sum_{b\in\gC_1}X_b^{(n)}(w)=\sum_{i=1}^{n-1}X_{i,i+1}^{(n)}(w)={\rm des}(w)$,
  applying Lemma \ref{lemma:1} with $\gd=\e/n$ gives us
  \begin{equation}\label{eq:B1}
    \Pr{{\rm des}(w)-\E{{\rm des}(w)}\geq\e}
    \leq 2\Exp{-\frac{\e^2(n-2)}{4n^2}}
  \end{equation}
  for any $\e>0$ where $w$ is sampled uniformly from $S_n^B$.

  Next, let $\gp_0,\gp_1,\gp_2,\ldots$ be a sequence of random
  \textit{signed} permutations, evolving such that given $\gp_{t-1}\in S_{t-1}^B$,
  we create $\gp_{t}$ by inserting $t$ at a (uniformly) random index in
  $\gp_{t-1}$, such that ${\mathrm{Pr}[\gp_t(i)=t\mid\gp_{t-1}]}=1/2t$
  for each $i\in [\pm t]$. Then $-t$ is inserted at the unique index that makes $\gp_{t}$ a signed permutation.
  Let $Q_t:=\mathrm{inv}_D(\gp_t)$.

  If $t$ is placed in a positive index $k$, then
  $\gp_t(i)<\gp_t(k)$ and $\gp_t(-i)<\gp_t(k)$ for all $1\leq i<k$, and
  $\gp_t(k)>\gp_t(j)$ and $\gp_t(-k)<\gp_t(j)$ for all $j$ such that $k<j\leq t$.
  This leads to a total increase of $t-k$ to $Q_{t-1}$.
  If instead $-t$ is placed at a positive index $k$, then
  $\gp_t(i)>\gp_t(k)$ and $\gp_t(-i)>\gp_t(k)$ for all $1\leq i<k$, and
  $\gp_t(k)<\gp_t(j)$ and $\gp_t(-k)>\gp_t(j)$ for all $j$ such that $k<j\leq t$.
  This leads to a total increase of $2(k-1)+(t-k)=t+k-2$ to $Q_{t-1}$.

  The two possibilities are equally likely, as are their indices.
  For $k_+,k_-\in[t]$, $t-k_+=t+k_--2$ only if $t-k_+=t+k_--2=t-1$. Thus
  \begin{align*}
    \sPrg{Q_t=Q_{t-1}+a}{Q_{t-1}}=
    \begin{cases}
      1/t&\text{if }a=t-1 \\
      1/2t&\text{if }a\in(\Set{0,\ldots, 2(t-1)}\setminus\Set{t-1}) \\
      0&\text{otherwise}
    \end{cases}
  \end{align*}
  for each $t$.
  Thus,
  \begin{align*}
    \sPrg{Q_{t+1}=Q_{t}+a}{Q_{t}}=
    \begin{cases}
      \frac{1}{t+1}&\text{if }a=t \\
      \frac{1}{2(t+1)}&\text{if }a\in(\Set{0,\ldots,2t}\setminus\Set{t}) \\
      0&\text{otherwise}
    \end{cases}
  \end{align*}
  and so $\sPrg{Q_{t+1}=Q_{t}+a}{Q_{t}}=\sPrg{Q_{t+1}=Q_{t}+2t-a}{Q_{t}}$.
  Then setting $c_t=2t$, Lemma \ref{lemma:2} gives
  \begin{equation}\label{eq:preB2}
    \Pr{\E{Q_n}-Q_n\geq\e}
    \leq \Exp{-\frac{3\e^2}{n(n-1)(2n-1)}}
  \end{equation}
  for any $\e>0$.

  \noindent Notice that ${\rm inv}_D(w_1)$ and ${\rm inv}_D(\iv{w_2})$ are identically distributed when
  $w_1$ and $w_2$ are independently sampled uniformly at random from $S_n^B$ since
  ${\rm inv}_D(\iv w)={\rm inv}_D(w)$ for each $w\in S_n^B$.
  If $\gp_t$ is the evolved signed permutation after $t$ insertions,
  then $\gp_{t-1}$ is completely determined by $\gp_t$.
  Further, there is only one insertion of $t$ that evolves $\gp_{t-1}$ into
  $\gp_t$.
  The conclusion, via induction, is that the probability of seeing $\gp_t$ after
  $t$ insertions is $1/(2^tt!)$.
  The above implies that $Q_n$ has an identical distribution to
  ${\rm inv}_D(\iv w)$, and hence ${\rm inv}_D(w)$, when $w$ is sampled uniformly at random from
  $S_n^B$.
  This, combined with \eqref{eq:preB2}, implies that for $w$ sampled uniformly from $S_n^B$
  \begin{equation}\label{eq:B2}
    \Pr{\E{{\rm inv}_D(w)}-{\rm inv}_D(w)\geq\e}
    \leq \Exp{-\frac{3\e^2}{n(n-1)(2n-1)}}
  \end{equation}
  for any $\e>0$.

  Finally, take $f_1$ and $f_2$ to be functions over the reals with
  $f_1(x)=x$ and $f_2(x)=(x+1)^2$.
  Also, let $r=n$, $\e_1=n^2/16$, and $\e_2=n/16$.
  Using Lemma~\ref{lemma:3}, for sufficiently large $n$, with
  \eqref{eq:B1} and \eqref{eq:B2} gives \[
  \begin{split}
      \#\Set{w\in S_n^B\mid \mathrm{inv}_D(w)\leq n+(\mathrm{des}(w)+1)^2} \qquad\qquad\qquad\qquad\qquad\qquad\qquad\qquad\qquad\\
      \begin{aligned}      
        &=|S_n^B|\cdot\mathop{\mathrm{Pr}}\limits_{w\gets S_n^B}[\mathrm{inv}_D(w)\leq n+(\mathrm{des}(w)+1)^2] \\
        &\leq 2^nn!\cdot\left[2\Exp{-\frac{n-2}{256}} + \Exp{-\frac{3n^3}{256(n-1)(2n-1)}} \right] \\
        &\leq 2^nn!e^{-\gO(n)}.
      \end{aligned}
  \end{split}
  \]\end{proof}

\begin{proposition}\label{prop:bn}
  The number of proper elements in the Coxeter group of type $B_n$ is at most
  $2^nn!e^{-\gO(n)}$.
\end{proposition}
\begin{proof}
  We have that \[
      \begin{aligned}
        \#\Set{w\in S_n^B\mid w\text{ is proper}}
        &=\#\Set{w\in S_n^B\mid\mathrm{inv}_B(w)\leq n+(\mathrm{des}_B(w))^2} \\
        &\leq\#\Set{w\in S_n^B\mid\mathrm{inv}_B(w)\leq n+(\mathrm{des}(w)+1)^2} \\
        &\leq \#\Set{w\in S_n^B\mid\mathrm{inv}_D(w)\leq n+(\mathrm{des}(w)+1)^2} \\
        &\leq 2^nn!e^{-\gO(n)},
      \end{aligned}
  \]
where the first step follows by Proposition \ref{prop:maxwAnalysis}(ii), the second by $\mathrm{des}(w)+1\geq\mathrm{des}_B(w)$, the third by $\mathrm{inv}_D(w)\leq\mathrm{inv}_B(w)$, and the last by Proposition \ref{prop:snb}.
\end{proof}

\begin{corollary}
\label{cor:asymptoticB}
  The proportion of proper elements in the Coxeter group of type $B_{n}$
  vanishes as $n$ goes to infinity.
\end{corollary}
\begin{proof}
  Since $|S_n^B|=2^nn!$, it follows from Proposition \ref{prop:bn} that \[
    \mathop{\mathrm{Pr}}\limits_{w\gets S_n^B}[w \text{ is proper}]
    \leq e^{-\gO(n)},
  \]
  which tends to zero as $n$ goes to infinity.
\end{proof}

\begin{proposition}\label{prop:dn}
  The number of proper elements in the Coxeter group of type $D_n$ is at most
  $2^{n}n!e^{-\gO(n)}$.
\end{proposition}
\begin{proof}
  We have that \[
      \begin{aligned}
        \#\Set{w\in S_n^D\mid w\text{ is proper}}
        &=\#\Set{w\in S_n^D\mid\mathrm{inv}_D(w)\leq n+{\sf maxw}_0(S_n^D,\mathrm{des}_D(w)))} \\
        &\leq \#\Set{w\in S_n^D\mid\mathrm{inv}_D(w)\leq n+(\mathrm{des}_D(w))^2} \\
        &\leq\#\Set{w\in S_n^D\mid\mathrm{inv}_D(w)\leq n+(\mathrm{des}(w)+1)^2} \\
        &\leq \#\Set{w\in S_n^B\mid\mathrm{inv}_D(w)\leq n+(\mathrm{des}(w)+1)^2} \\
        &\leq 2^nn!e^{-\gO(n)},
      \end{aligned}
  \]
  where the first step follows by the definition of proper, the second by Proposition \ref{prop:maxwAnalysis}(iii), the third by $\mathrm{des}(w)+1\geq\mathrm{des}_D(w)$, the fourth by $S_n^D\!\subset\!S_n^B$, and the last by Proposition \ref{prop:snb}.
\end{proof}

\begin{corollary}
\label{cor:asymptoticD}
  The proportion of proper elements in the Coxeter group of type $D_{n}$
  vanishes as $n$ goes to infinity.
\end{corollary}
\begin{proof}
  Since $|S_n^D|=2^{n-1}n!$, it follows from Proposition \ref{prop:dn} that \[
    \mathop{\mathrm{Pr}}\limits_{w\gets S_n^D}[w \text{ is proper}]
    \leq e^{-\gO(n)},
  \]
  which tends to zero as $n$ goes to infinity.
\end{proof}

\subsection{Type I}
The type $I_2(n)$ Coxeter group is presented with generators $S=\Set{s_1,s_2}$
which satisfy $s_1^2=s_2^2=(s_1s_2)^n=e$.
\begin{proposition}
\label{prop:asymptoticI}
  Let $W_n$ be a Coxeter group of type $I_2(n)$
  and let $w$ be drawn from $W_n$ uniformly at random.
  Then $\Pr{w\text{ is proper}}\longrightarrow0$ as
  $n\longrightarrow\infty$.
\end{proposition}
\begin{proof}
  We know that $W_n$ has order $2n$.
  Further, for $w\in W_n$, \[
    d(w)=
    \begin{cases}
      2& \text{if }w=w_0, \\
      0& \text{if }w=e, \\
      1&\text{otherwise}.
    \end{cases}
  \]
  \begin{itemize}
    \item
      If $w=w_0$, then ${\sf maxw}_0(W_n, d(w))={\sf maxw}_0(W_n, 2)=n$.
      Thus, $w$ is proper if and only if $\ell(w)\leq 2+n$,
      which is true for all $n$ since $\ell(w_0)=n$.
    \item
      If $w=e$, then ${\sf maxw}_0(W_n, d(w))={\sf maxw}_0(W_n, 0)=0$.
      Thus $w$ is proper if and only if $\ell(w)\leq2$,
      which is true since $\ell(e)=0$.
    \item
      If $w\neq e,w_0$, then ${\sf maxw}_0(W_n, d(w))={\sf maxw}_0(W_n, 1)=1$.
      Thus, $w$ is proper if and only if $\ell(w)\leq3$.
      There are at most a constant number of $w\in W_n$ with length at most 3:
      $w=s_1,s_2,s_1s_2,s_2s_1,s_1s_2s_1$, and $s_2s_1s_2$.
  \end{itemize}
  Thus the number of proper elements in $W_n$ is constant with respect to $n$.
\end{proof}

\subsection{Proof of Theorem~\ref{thm:main1}} Theorem ~\ref{thm:main1} now follows immediately from Corollaries~\ref{cor:asymptoticA}, \ref{cor:asymptoticB}, and \ref{cor:asymptoticD} and Proposition~\ref{prop:asymptoticI}.

\section{Lower bounds}
\label{sec:lowerbounds}
\subsection{Types $A_{n-1}$, $B_n$, and $D_n$} We have given upper bounds on the number of proper elements in
each of the infinite families of finite irreducible Coxeter groups.
Next, we give a non-trivial lower bound on the number of proper elements
for families $A_{n-1}$, $B_n$, and $D_n$.
In order to show these lower bounds,
we construct a subset of $S_n$ such that
all elements of the set are proper.
Further, we can use this subset to create
subsets of $S_n^B$ and $S_n^D$
that contain only proper elements.
This will give us a lower bound on the number of elements
in each of the groups $A_{n-1}$, $B_n$, and $D_n$.

The next lemma gives mappings that, when applied to each
element of a set of proper elements of $S_n$,
will result in a set that only contain proper elements of $S_n^B$ and a set that only contain proper elements of $S_n^D$.

\begin{lemma}\label{lemma:containment}
  Let $\fun{\gf_B}{S_n}{S_n^B}$ and $\fun{\gf_D}{S_n}{S_n^D}$
  be the homomorphisms $w\mapsto w'$ where $w'(i)=w(i)$ and $w'(-i)=-w(i)$ for
  each $i\in[n]$.
  For $n$ large enough and for all $w\in S_n$,
  if $w$ is proper then $\gf_B(w)$ and $\gf_D(w)$ are proper.
\end{lemma}
\begin{proof}
  We have
  \begin{equation}
  \label{eq:contain1}
  \ell_B(\gf_B(w))=\ell_D(\gf_D(w))=\ell_A(w)
  \end{equation}
  since $w(i)>0$ for each $i\in[n]$, by the definitions of $\mathrm{inv}$, $\mathrm{inv}_B$, and $\mathrm{inv}_D$.
  Further, $\iv{(\gf_B(w))}=\gf_B(\iv w)$ and $\iv{(\gf_D(w))}=\gf_D(\iv w)$ implies that
  \begin{equation}
  \label{eq:contain2}
  d_B(\gf_B(w))=d_D(\gf_D(w))=d_A(w).
  \end{equation}
  If $w\in S_n$ is proper, then $\ell_A(w)\leq (n-1)+\binom{d_A(w)+1}{2}\leq n+m_A(w)$ by Proposition \ref{prop:maxwAnalysis}.
  Thus, since ${\sf maxw}_0(A_{n-1}, d_A(w))\leq \min\Set{{\sf maxw}_0(B_n, d_B(\gf_B(w))),{\sf maxw}_0(D_n, d_D(\gf_D(w)))}$ for large enough $n$, \eqref{eq:contain1}, \eqref{eq:contain2}, and Proposition \ref{prop:maxwAnalysis} imply our result.
\end{proof}

Next, we give a construction of a subset of elements of $S_n$
that, with the optimized parameters, are all proper.
To construct an element of $S_n$,
we first imagine we have a row of $n$ empty cells.
We group the cells into $a$ contiguous \textit{chunks} of $q$ cells.
There might be some $r < q$ cells leftover on the rightmost side of the row.
Let us focus on one chunk and denote this chunk $\alpha$.
We will place the numbers $1,\ldots,q$ into chunk $\alpha$ such that
there are many descents.
More specifically, we group the numbers $1,\ldots,q$ into $s$ contiguous
(when the numbers are arranged in increasing order)
\textit{runs}
of size $b$.
There might be $d<b$ remaining larger numbers.
We will focus on one run of $1,\ldots,q$ and call this run $\beta$.
We place the elements of run $\beta$ into the empty cells of chunk $\alpha$
such that when looking at the cells from left to right, the elements
of $\beta$ are in reverse (decreasing) order.
We do this for each run $\beta$ and place the excess elements of $\beta$
in the remaining empty cells arbitrarily.
Then we add $\alpha\cdot q$ to each cell entry of chunk $\alpha$ so that globally we do not have any repeated numbers in the cells.
We repeat this process for all chunks.
The remaining empty cells are filled in increasing order with the remaining elements of $[n]$ that
have not been used.

The following is a formalization of the above intuition.

\begin{construction}\label{construction}
  Let $q,s\in \Z_{\geq0}$ be such that $n=aq+r$ for some unique $a\in\Z$,
  $0\leq r<q$, and $q=sb+d$ for some $b\in\Z$, $0\leq d < s$.
  Let $\gp_0,\ldots,\gp_{a-1}\in S_q$ with the property that \[
    \iv{\gp_j}(st+1) > \iv{\gp_j}(st+2) > \cdots > \iv{\gp_j}(st+s)
  \]
  for each $t=0,\ldots,b-1$ and $j=0,\ldots,a-1$.
  Now set $w\in S_n$ to be such that for $i=\ga q+\gb\in[n]$ with $\ga\in\Z$ and
  $\gb\in[q]$, \[
    w(i) = \ga q +
    \begin{cases}
      \gp_\ga(\gb) &\text{if } \ga < a, \\
      \gb          &\text{otherwise}.
    \end{cases}
  \]
  Let $P_n$ be the set of these constructed permutations on $[n]$.
\end{construction}

We note that the above construction results in proper elements of $S_n$ only when
the chunk sizes and runs are optimized correctly.
We give some examples of this construction.

\begin{example}
  Some examples in one-line notation are listed below.
  \begin{enumerate}
    \item $w=2143\,6587\in P_n$ with $n=8$, $q=4$, $s=2$ and
          $a=2$, $r=0$, $b=2$, $d=0$.
          In this example, $\gp_0=\gp_1=2143$.
    \item $w=2431\,6587\,9\in P_n$ with $n=9$, $q=4$, $s=2$ and
          $a=2$, $r=1$, $b=2$, $d=0$.
          In this example, $\gp_0=2431$ and $\gp_1=2143$.
    \item $w=453261\,789\in P_n$ with $n=9$, $q=6$, $s=4$ and
          $a=1$, $r=3$, $b=1$, $d=2$.
          In this example, $\gp_0=453261$.
  \end{enumerate}
\end{example}

By way of Construction~\ref{construction}, we force the number of descents to be somewhat large within each chunk (and thus globally), and we force the inversions to occur only within each chunk which makes the number of inversions somewhat small.
In the following theorem, we show that there always exists chunk and run sizes that result in $P_n$ containing only proper elements of $S_n$. Further, we give lower bounds for the sizes of $P_n$
and the corresponding constructed sets for $S_n^B$ and $S_n^D$.

\begin{theorem}
\label{thm:lowerBound}
  The number of proper elements of $S_n$, $S_n^D$, and $S_n^B$ is at least
  $(cn)^n$ for some absolute constant $c>0$.
\end{theorem}
\begin{proof}
    We will show there
    exist $q,s,a$, and $b$ from Construction~\ref{construction} such that each element of $P_n$ is proper.
    We can assume $n$ is sufficiently large, since for smaller $n$, we can increase $c$ such that the statement holds.
    Let $a=10$, $s=5$, $q=\Floor{n/a}$, $b=\Floor{q/s}$ such that $q,b>0$ for sufficiently large $n$.
    % $nq/2\leq n+(a(s-1)b)^2/2$.
    Fix $w\in P_n$.
    Since inversions of $w$ only take place within each of the $a$ chunks,
    we have $\ell_A(w)\leq a\binom{q}2\leq aq^2/2 \leq 10\cdot (n/10)^2/2 = n^2/20$. 
    Also, in each chunk of $w$, there are at least $s$ runs of size $b$ placed in descending order (and not necessarily consecutively);
    this means $d_A(w)\geq asb$.
    Thus $\binom{d_A(w)+1}2\geq (d_A(w))^2/2\geq (asb)^2/2 \geq (50b)^2/2 \geq (n-60)^2/2$.
    This means that for $n$ sufficiently large, $\ell_A(w) \leq n + \binom{d_A(w)+1}2$ and hence $w$ is proper.
    
    Further, we have that $|P_n|\geq\frac{(q!)^a}{(s!)^{ab}}$
    by the constraints of Construction~\ref{construction}.
  % Thus for $q=\Floor{n/10}$, $a=10$, $s=5$, $b=\Floor{q/5}$, $\e>0$, and sufficiently
  For $\e>0$ and sufficiently large $n$,
  \begin{align*}
    \frac{(q!)^a}{(s!)^{ab}}
    &\geq\frac{(\Floor{n/a}!)^{a}}{(s!)^{n/s}} \\
    &\geq\left(\frac{\Floor{n/a}^{\Floor{n/a}}}{e^{\Floor{n/a}-1}}\right)^{a}%
      \cdot\left(\frac{e^{s-1}}{s^{s+1}}\right)^{n/s} \numberthis \label{ineq:knuth} \\% \tag{$\bigstar$} \\
    &\geq \left(\frac{n}{(a+\e)\cdot e}\right)^{n-a}%
      \cdot\left(\frac{e^{s-1}}{s^{s+1}}\right)^{n/s}\cdot e^{-a} \numberthis \label{ineq:floor} \\% \tag{$\blacklozenge$} \\
    &=(\gamma_n \cdot n)^n,
    %\geq\left(\frac{(n/11)^{n/11}}{e^{n/9}}\right)^{10}%
    %  \cdot\left(\frac{e^4}{5^6}\right)^{n/5} \\
    %&\geq\left(\frac{e^{59/45}}{n^{1/11}}\right)^n\cdot\left(\frac{n}{62e}\right)^n
    % =\left((e^{14/45}/62)\cdot n^{10/11}\right)^n
    % =((0.022\ldots) \cdot n^{1-1/11})^n.
  \end{align*}
  where \[
    \gamma_n = \inv{(a+\e)\cdot e^{1/s}\cdot s^{1+1/s}}\cdot\gn_n
    \quad \text{and}\quad
    \gn_n=\left(\frac{a+\e}{n}\right)^{a/n}.
  \]
  Step \eqref{ineq:knuth} uses the bound \[
    n^n/e^{n-1}\leq n!\leq n^{n+1}/e^{n-1}
  \] from \S 1.2.5 Exercise 24 of \cite{K97},
  and step \eqref{ineq:floor} uses the bound
  $\Floor{n/a}\geq n/a-1\geq \frac{n}{a+\e}$ for $n\geq a\cdot(1+a/\e)$. % $n\geq \frac{a(a+\e)}{\e}$.
  Because $\log \gn_n \geq -\frac{a}{(a+\e)\cdot e}$ by elementary calculus,
  there is some absolute constant $c>0$ such that $\gamma_n \geq c$.
  % One can verify that using $a=10$ and $s=5$ satisfies the requirements of Construction \ref{construction}.

  This is a non-trivial lower bound for the number of proper elements of $S_n$.
  The rest of the proof follows from Lemma \ref{lemma:containment}.
\end{proof}

\section{The exceptional groups}
\label{sec:exceptional}
\subsection{Types $E_6$, $E_7$, $E_8$, $F_4$, $H_3$, and $H_4$}
The Coxeter groups of types $E_6$, $E_7$, $E_8$, $F_4$, $H_3$, and $H_4$ are finite families.
Figure \ref{fig:finite-families} lists the number of proper elements for each of these groups.
Using a computer we computed all of their proper elements \cite{B21}.
The algorithm we used is described below.

\begin{enumerate}
\item \label{alg:gen1}(Generating the Coxeter group elements by length).
Let $W=\langle s_1,\ldots,s_n\rangle$ be a finite Coxeter group and let $l \in \Set{0,\ldots, \ell(w_0(W))}$.
At step $t$ of the algorithm, we maintain a set $Z_t \subseteq W$ with $Z_0=\Set{e}$.
Then $Z_{t+1} := \Set{s_iw\mid w\in Z_t, s_i\in [n]\setminus J(w)}$.
We terminate at the first $T$ such that $Z_{T+1} = \emptyset$.
Then for $Z := \cup_{t=0}^{T} Z_t$, we collect the elements $Y := \Set{w\in Z\mid \ell(w)=l}$.

\item \label{alg:gen2}(Storing the Coxeter group elements by length).
The elements of $Y$ are written to a named file (e.g. \texttt{data/E8/23.txt}, where $l=23$ in this example).
Each line of this file is a reduced word of $W$ with length $l$, represented by the ordered indices of the generators.
Then we compress the directory for each Coxeter group name.  

\item\label{alg:gen3} (Computing the number of proper elements by length).
For each Coxeter group, we (in parallel) extract each of its files and compute the the number of proper elements. Then we take the sum of these results.
\end{enumerate}

% TODO: potentially edit this section.
The computation for $E_8$ took the longest: it took approximately three weeks to run steps \eqref{alg:gen1} and \eqref{alg:gen2}.
Step \eqref{alg:gen3} took approximately 75 minutes to run.
We used a 3.0GHz 5th Gen Intel Core i7-5500U processor with 1TB of disk space and 16GB of RAM. 
Steps \eqref{alg:gen1} and \eqref{alg:gen3} heavily utilize \texttt{SageMath} and \texttt{Coxeter3} libraries
for most Coxeter group operations.

\begin{figure}
  \renewcommand{\arraystretch}{1.55}
  \begin{tabular}{cc}
    Type & Number of proper elements \\
    \hline
    $E_6$ & 10690 \\
    $E_7$ & 159368 \\
    $E_8$ & 4854344 \\
    $F_4$ & 297 \\
    $H_3$ & 48 \\
    $H_4$ & 594
  \end{tabular}
  \caption{The number of proper elements for the finite families were calculated using the aid of a computer \cite{B21}.}
  \label{fig:finite-families}
\end{figure}

\section*{Acknowledgements}
We thank Alex Yong for helpful discussions.
We thank the anonymous referees for their helpful remarks on the organization of this paper.
The project was completed as part of the ICLUE
(Illinois Combinatorics Lab for Undergraduate Experience) program,
which was funded by the NSF RTG grant DMS 1937241.
Research was partially supported by NSF Grant DMS 1764123,
NSF RTG grant DMS 1937241,
Arnold O. Beckman Research Award (UIUC Campus Research Board RB 18132),
the Langan Scholar Fund (UIUC), and the Simons Fellowship.
Part of this work was completed while RH was a postdoc at the University of Illinois at Urbana-Champaign.
RH was partially supported by an AMS Simons Travel grant.

\section*{Conflicts of interest} The authors declare no conflicts of interest.

\end{document}